\documentclass[12pt, a4paper]{article}
\usepackage[margin=1in]{geometry}
\usepackage{amsmath, amsthm, amscd, amsfonts, amssymb, graphicx, color,float,listings,ragged2e,authblk, tikz}
\usepackage[shortlabels]{enumitem}

\newtheorem{thm}{Theorem}[section]
\newtheorem{prop}[thm]{Proposition}
\newtheorem{def1}[thm]{Definition}
\newtheorem{rem}[thm]{Remark}
\newtheorem{eg}[thm]{Example}
\newtheorem{cor}[thm]{Corollary}

\newtheorem{obs}[thm]{Observation}
\newtheorem{qn}[thm]{Question}

\newcommand{\dend}{\mathop{\overrightarrow{\mathrm{EG}}}\nolimits}
\newcommand{\uend}{\mathop{\mathrm{EG}}\nolimits}
\newcommand{\dende}{\mathop{\overrightarrow{\mathrm{EG^*}}}\nolimits}
\newcommand{\uende}{\mathop{\mathrm{EG^*}}\nolimits}
\newcommand{\daut}{\mathop{\overrightarrow{\mathrm{AG}}}\nolimits}

\newcommand{\dpow}{\mathop{\overrightarrow{P}}\nolimits}
\newcommand{\Ker}{\mathop{\mathrm{Ker}}}
\newcommand{\End}{\mathop{\mathrm{End}}}
\newcommand{\Aut}{\mathop{\mathrm{Aut}}}
\newcommand{\Dic}{\mathop{\mathrm{Dic}}\nolimits}
\newcommand{\Z}{\mathbb{Z}}

\lstdefinestyle{mystyle}{
        basicstyle=\ttfamily\small,
        breakatwhitespace=false,
        breaklines=true,}

\def \ni{\noindent}

\author[1]{Midhuna V Ajith \footnote{E-mail: midhunavajith@gmail.com}}
\author[2]{Peter J Cameron \footnote{E-mail: pjc20@st-andrews.ac.uk}}
\author[3]{Mainak Ghosh \footnote{E-mail: mainak.09.13@gmail.com}}
\author[4]{Aparna Lakshmanan S \footnote{E-mail: aparnals@cusat.ac.in, aparnaren@gmail.com}}
\affil[1, 4]{Department of Mathematics\\
        Cochin University of Science and Technology\\Cochin -
        22}
\affil[2]{School of Mathematics and Statistics, University of St. Andrews,
Fife, UK}
\affil[3]{Dept. of Mathematics, Indian Institute of Science, Bangalore, INDIA}

\begin{document}

\title{Endomorphism and automorphism graphs of finite groups}
\date{}
\maketitle

\begin{abstract}
Let $G$ be a group. The directed endomorphism graph, $\dend(G)$ of $G$ is a directed graph with vertex set $G$ and there is a directed edge from the vertex $a$ to the vertex $b$ if $a \neq b$ and there exists an endomorphism on $G$ mapping $a$ to $b$. The endomorphism graph, $\uend(G)$ is the corresponding undirected simple graph. The automorphism graph  of $G$ is similarly defined for automorphisms: it is a disjoint union of complete graphs on the orbits of $\Aut(G)$.

The endomorphism digraph is a special case of a digraph associated with a
transformation monoid, and we begin by introducing this.

We have explored graph theoretic properties like size, planarity, girth etc. and tried finding out for which types of groups these graphs are complete, diconnected, trees, bipartite and so on, as well as computing these graphs for some
special groups.

We conclude with examples showing that things are not always simple.\\
\ni\line(1,0){395}\\
\ni {\bf Keywords:} endomorphism, automorphism, graph, digraph, transformation
monoid\\

\ni {\bf AMS Subject Classification:} 05C20, 05C25, 08A35, 20M20\\
\ni\line(1,0){395}
\end{abstract}

\section{Introduction}

The subject of graphs on groups is of undoubted importance. The oldest such
graphs, Cayley graphs, date from the 19th century, and are very significant
in both algebraic combinatorics and geometric group theory, the latter in the
context of hyperbolic groups as defined by Gromov~\cite{Gromov}. They are also
related to other areas of mathematics such as regular maps~\cite{Conder}, and
are used as classifiers for data mining~\cite{kelarev}.

More relevant to the present study are graphs which are defined just in terms
of the group structure; the vertices are group elements and edges are defined
by a group-theoretic property. The oldest of these, from 1955, is the commuting
graph, in which the vertices are group elelments, two vertices joined if they
commute. This was introduced by Brauer and Fowler~\cite{com} in 1955 in their
proof that there are only finitely many finite simple groups with a given
involution centralizer (a result which could be regarded as the first step
towards the classification of the finite simple groups). Several other graphs,
including the generating graph and the power graph, have been defined and
studied. The power graph makes a brief appearance below; we will define it
there.

The graphs studied here are defined on a group $G$ in terms of the 
endomorphisms and automorphisms of $G$.

\begin{def1}\rm
An \textbf{endomorphism} of a group $G$
is a map $f:G\to G$ such that $(x*y)^f=x^f*y^f$. (We write maps on the right
of their argument so that they compose left-to-right.) 

The \textbf{endomorphism digraph} $\dend(G)$ of a group $G$ takes the
vertex set to be $G$ with an arc from $x$ to $y$ if some endomorphism of $G$
maps $x$ to $y$. The \textbf{endomorphism graph} $\uend(G)$ is obtained by
ignoring the directions and suppressing double edges that result.

There are \textbf{compressed} versions of these, obtained by deleting the
identity and shrinking each automorphism class to a single vertex. We denote
these by $\dend_-(G)$ and $\uend_-(G)$ respectively.
\\
A point basis in a digraph is a  set such that every vertex receive an incoming arc from at least one vertex in the set (or be part of the set itself).
\end{def1}

In this paper, we first introduce these graphs as a special case of graphs
associated with transformation monoids, which enables us to conclude that
the undirected graphs are perfect. Then we explore properties such as planarity
and girth, and compute the graphs for certain special groups.

Our notation is fairly standard. The endomorphisms of $G$ form a transformation monoid on $G$ called the
\textbf{endomorphism monoid} $\End(G)$ of $G$; its units are the automorphisms
of $G$ and form the \textbf{automorphism group} $\Aut(G)$ of $G$. The
identity element of $G$ is denoted by $e_G$ (or just $e$ if the group is clear).
Transformation monoids form a natural context for our work, so we look at these
first before concentrating on endomorphism graphs. Then we examine properties
of endomorphism graphs such as planarity and girth
and calculate their structure for specific groups such as dihedral, dicyclic and symmetric groups and abelian $p$-groups.

\section{Transformation monoids, preorders, and digraphs}

The general context for our construction is that of a \textbf{transformation monoid} on a set $X$; this is a collection of maps from $X$ to $X$ closed under composition and containing the identity map. Thus, it is a monoid (a semigroup with identity). If $X$ is a finite set, then the units of a transformation monoid are the permutations it contains; these form a permutation group, which defines an orbit partition on $X$.

Let $M$ be a transformation monoid on $X$. Define a relation $\to$ by the rule that $x\to y$ if there is an element of $M$ which maps $x$ to $y$. (Here we allow $x=y$). This relation is reflexive and transitive; thus it is a \textbf{partial preorder}. It is also called a \textbf{preferential arrangement}, since it describes a situation where someone has preferences among a set of choices but may be indifferent to some pairs in the sense that neither is preferred to the other.

Let $\to$ be a partial preorder on $X$. Define a relation $\equiv$ on $X$ by $x\equiv y$ if both $x\to y$ and $y\to x$ hold. Then $\equiv$ is an equivalence relation on $X$, and so defines a partition of $X$; the equivalence classes are called \textbf{indifference classes}, and the partial preorder induces a partial order on the set of equivalence classes. 

A partial preorder gives rise to a directed graph $\vec\Gamma$ on $X$, whose arcs are the pairs $(x,y)$ with $x\to y$ and $x\ne y$. We also obtain an undirected graph by ignoring the directions (and placing a single edge $\{x,y\}$ if $x\to y$ and $y\to x$). The graph of a partial preorder is a \textbf{perfect graph} (that is, all its induced subgraphs have clique number equal to chromatic number). This is because we can turn a partial preorder into a partial order by simply imposing a total order on each indifference class; then $\Gamma$ is the comparability graph of this partial order, and it is perfect by Mirsky's theorem \cite{Mirsky}.

In the case of a transformation monoid $M$, the orbit partition of the group of units refines the partition into indifference classes for the partial preorder, though they are not equal in general.

If $A$ and $B$ are indifference classes for $M$, and $a\to b$ for some $a\in A$ and $b\in B$, then $a'\to b'$ holds for all $a'\in A$ and all $b'\in B$. For there exist endomorphisms carrying $a'$ to $a$, $a$ to $b$, and $b$ to $b'$; their composition maps $a'$ to $b'$. The same hold if we replace indifference classes by unit group orbits. This suggests that we study the simpler digraph and graph obtained by contracting the indifference classes or the unit group orbits to single orbits.

\section{The endomorphism digraph of a group}

The set $\End{G}$ of endomorphisms of a group $G$ form a transformation monoid
on the set of elements of $G$; the directed endomorphism graph $\dend(G)$ is
just the digraph attached to the transformation monoid in the preceding
section. We also denote the diigraph whose edges come from automorphisms of $G$ by
$\daut(G)$.

In a group, every element can be mapped to the identity $e$ by an endomorphism, but $e$ cannot be mapped to any non-identity element. So there are arrows $a\to e$ but no arrows $e\to a$ for every $a\ne e$. Accordingly, we lose no information by deleting the identity element of the group. Let $[a]_e$ denote the indifference class of the element $a$ when the transformation monoid under consideration are the endomorphisms of a group. If $a\to b$ in the
endomorphism digraph, then $a'\to b'$ for every $a'\in[a]_e$ and $b'\in[b]_e$. This
justifies our compression procedure. We will call the indifference classes of
the preorder the \textbf{endomorphism classes}: they form a partition of $G$
such that the \textbf{automorphism classes} (the orbits of the automorphism
group) form a refinement (in the sense that each automorphism class is
contained in a single endomorphism class).

Note that any isomorphism from $G$ to $H$ induces isomorphisms from the
various endomorphism digraphs on $G$ to those on $H$.
An isomorphism between compressed endomorphism digraphs $\dend_-(G)$ and
$\dend_-(H)$ is said to be \textbf{strong} if it maps each automorphism
class in $G$ to an automorphism class of the same size in $H$. Note that
strong isomorphisms are precisely those induced by isomorphisms of the
uncompressed digraphs. This suggests three questions:

\begin{qn}\label{q:first}\rm
\begin{enumerate}
\item Can there be an isomorphism from $\dend_-(G)$ to $\dend_-(H)$ which is
not strong? Can this happen if $|G|=|H|$?
\item Can non-isomorphic groups have isomorphic directed endomorphism graphs?
\item Is there a group for which endomorphism and automorphism classes do not coincide?
\end{enumerate}
\end{qn}

The first part of the first question has an easy negative answer: take cyclic
groups of different prime orders; the compressed isomorphism digraphs each have
just a single vertex. We will discuss the stronger form, and the answer to the
other two questions, in Section~\ref{s:examps}.

%

    \section{Direct products}
    \label{s:dp}
    The \textbf{strong product} of two digraphs (or graphs) $\Gamma$ and $\Delta$, $\Gamma \boxtimes \Delta$ on vertex sets $X$ and $Y$ is the graph whose vertex set is the Cartesian product $X\times Y$, with an arc (or edge) from $(x,y)$ to $(x',y')$ if one of the following holds:
    \begin{enumerate}
    \item $x\to x'$ (or $x\sim x'$) in $\Gamma$, $y=y'$;
    \item $x=x'$, $y\to y'$ (or $y\sim y'$) in $\Delta$;
    \item $x\to x'$ (or $x\sim x'$) in $\Gamma$, $y\to y'$ (or $y\sim y'$) in $\Delta$.
    \end{enumerate}

\begin{obs}\rm
Note that, if we take the strong product of digraphs $\Gamma$ and $\Delta$ and then symmetrise it by ignoring directions, we do not obtain the strong product of the symmetrisations of $\Gamma$ and $\Delta$ (If $\Gamma$ and $\Delta$ are the single arcs $x\to x'$ and $y\to y'$, the strong product has no arc between $(x,y')$ and $(x',y)$, but the strong product of the symmetrisations is the complete graph on four vertices).
\end{obs}

\begin{prop}
Let $G$ and $H$ be groups with coprime orders. Then the endomorphism monoid of $G\times H$ is $\End(G)\times\End(H)$, and the endomorphism digraph of $G\times H$ is the strong product of $\dend(G)$ and $\dend(H)$.
\end{prop}

For the proof, note that, if $f$ is an endomorphism of a group $G$, then the order of $g^f$ divides the order of $g$. So any endomorphism of $G\times H$ must map $G$ and $H$ to themselves, and so induce endomorphisms of $G$ and $H$. It is easy to see that this map is a bijection, so the result about endomorphisms is proved. From this it follows that the endomorphism digraph of $G\times H$ is the strong product of the endomorphism digraphs of $G$ and $H$.

\begin{obs}\rm
The procedure of collapsing automorphism classes commutes with this isomorphism, but it does not follow that $\dend_-(G\times H)$ is isomorphic to $\dend_-(G)\,\boxtimes\,\dend_-(H)$: for automorphism classes $([e_G]\times [h]$ and $[g]\times[e_H]$ occur as vertices in $\dend_-(G\times H)$ but not in the strong product. However, it is true that if
$\dend(G_1)\simeq\dend(G_2)$ and $\dend(H_1)\simeq\dend(H_2)$, then $\dend(G_1\times H_1)\simeq\dend(G_2\times H_2)$.
\end{obs}
     
\section{Cyclic groups}\label{s:cyc}

We illustrate these concepts by considering the cyclic group $Z_n$ of order~$n$.
We denote the elements of this group by $\{0,1,\ldots,n-1\}$, the group
operation being addition modulo~$n$.

An endomorphism $f$ is uniquely  determined by the image of $1$; if $1^f=a$
then $x^f=xa$ (mod~$n$) for all $x$. We denote this endomorphism by $f_a$. If
$\gcd(a,n)=d$, then the image of $f_a$ consists of all multiples of $d$.
Moreover, since $f_af_b=f_{ab}$ (subscript mod~$n$), we see that $f_a$ is a
unit if and only if $\gcd(a,n)=1$; and more generally, $f_a$ generates the
same endomorphisms as $f_d$, where $\gcd(a,n)=d$. Note that there is an
endomorphism mapping $x$ to $y$ if and only if $\gcd(x,n)$ divides $\gcd(y,n)$.
Hence the following are equivalent:
\begin{itemize}
\item $x$ and $y$ lie in the same endomorphism class;
\item $\gcd(x,n)=\gcd(y,n)$;
\item there is an automorphism mapping $x$ to $y$.
\end{itemize}
So the endomorphism and automorphism classes coincide. In particular, we see
that any automorphism class contains a unique divisor of $n$; in what follows,
we use this element to label the class. So the compressed endomorphism digraph
has vertices indexed by divisors of $n$ (except $1$), with an edge $[x]\to[y]$
whenever $x$ divides $y$. In particular, $\uend_-(Z_n)$ is complete if and
only if the divisors of $n$ form a chain, that is, if and only if $n$ is a
prime power.

Moreover, the number of elements in a class $[a]$ is $\phi(|a|)$, which is
$\phi(n/a)$ if we choose $a$ to be a divisor of $n$. (Here $\phi$ is Euler's
function.)

\begin{obs}\rm
The \textbf{directed power graph} $\vec{P}(G)$ of a finite group $G$ has
vertex set $G$, with an arc $x\to y$ whenever $y$ is a power of $x$. From the
description above, we see that $\dend(Z_n)$ is the same as $\vec{P}(Z_n)$. It
is known~\cite{c:pg} that groups with isomorphic power graphs have isomorphic directed
power graphs (where the power graph $P(G)$ is obtained from $\vec{P}(G)$ by
ignoring directions and collapsing multiple edges); but this does not hold
in general for the undirected and directed endomorphism graphs, as we will see.
\end{obs}
	
\begin{thm}
Let $G$ be a cyclic group of order $n$ and $1 < d_1 \leq d_2 \leq \ldots \leq d_k <n$ be the divisors of $n$. Then the total number of edges in $\uend(G)$ is
\begin{equation}
\binom{n}{2} - \sum_{\substack {1 \leq i < j \leq k\\
d_i \nmid d_j}} \phi(d_i)\phi(d_j).
\end{equation}
\end{thm}

\begin{proof}
Let $a$ and $b$ be divisors of $n$. Then vertices in $[a]$ and $[b]$ are
joined if and only if one of $a$ and $b$ divides the other. So a pair $\{x,y\}$
is not an edge if and only if neither of the orders of $x$ and $y$ divides
the other; there are $\phi(d_i)\phi(d_j)$ such pairs, if $d_i$ and $d_j$ are
the orders of $x$ and $y$.
\end{proof}

\begin{thm}
The number of maximal cliques in $\uend(\mathbb{Z}_n)$, where $n = {p_1}^{n_1}{p_2}^{n_2} \ldots {p_k}^{n_k}$ in which $p_1,p_2,\ldots,p_k$ are distinct primes and $n_i \in \mathbb{N}$, for every $i \in \{1,2,\ldots,k\}$ is $$\frac{(n_1+n_2+\cdots+n_k)!}{n_1!n_2!\cdots n_k!}$$.
\end{thm}

\begin{proof}
As noted, if we label isomorphism classes by divisors of $n$, then two classes
are joined if and only if the label of one divides the label of the other. So
the maximal cliques come from maximal chains in the lattice of divisors of $n$.
Any such maximal chain has the property that each element is obtained by
multiplying its predecessor by a prime. So the length of the chain is
$n_1+\cdots+n_k$, and we can form all possible chains by choosing subsets of
the positions in the sequence with cardinalities $n_1,n_2,\ldots,n_k$ and
inserting $p_i$ in all positions in the $i$th set; then the index of the $i$th
automorphism class in the chain is $p_1p_2\cdots p_i$.

So the number of maximal cliques is the multinomial coefficient
\begin{equation*}
\frac{(n_1+n_2+\cdots+n_k)!}{n_1!n_2!\cdots n_k!}
\end{equation*}
\end{proof}

\begin{thm}
Let $n=p_1^{n_1}\cdots p_k^{n_k}$, where $p_1,\ldots, p_k$ are distinct primes
and $n_1,\ldots,n_k$ positive integers. Write the primes $p_1,\ldots,p_k$ (with
multiplicities $n_1,\ldots,n_k$) in a sequence of length $n_1+\cdots+n_k$
where the primes are in nonincreasing order, say $(q_1,q_2,\ldots,q_m)$, where
$m=n_1+\cdots+n_k$, and let $r_i$ be the product of the first $i$ terms in this
sequence (with $r_0=1$). Then the clique number and chromatic number of
$\uend(Z_n)$ are equal to $\displaystyle{\sum_{i=0}^m\phi(r_i)}$.
\end{thm}

\begin{proof}
The proof of the preceding theorem shows that the size of a maximal clique 
is of the form, for some ordering of the $m$ primes (with their multiplicities)
dividing $n$. Now, if $p$ and $q$ are primes with $p<q$, then
$\phi(ph)<\phi(qh)$ for any positive integer $h$; so we can repeatedly swap
adjacent primes in the sequence to put them in nonincreasing order, and the
size of the clique will increase at each step.

The chromatic number is equal to the clique number since, as noted earlier,
the graph $\uend(G)$ is perfect for any group $G$.
\end{proof}

\begin{eg}\rm
Consider $G=Z_{12}$ (with addition mod~$12$).
\begin{figure}[H]
	\begin{center}
\begin{tikzpicture}[scale=0.15,style=thick,x=1cm,y=1cm]
			\def\vr{5pt}
			\begin{scope}[xshift=10cm, yshift=10cm] 				
                \coordinate(x0) at (0,-10);
				\coordinate(x1) at (-10,20);
				\coordinate(x2) at (-10,0);
				\coordinate(x3) at (10,0);
                \coordinate(x4) at (-20,10);
				\coordinate(x6) at (0,10);

                \draw(x0)[fill=black] circle(15pt) node[below]{ {\footnotesize $[0]=\{0\}$}};
				\draw(x1)[fill=black] circle(15pt) node[above]{ {\footnotesize $[1]=\{1,5,7,11\}$}};
				\draw(x2)[fill=black] circle(15pt) node[left]{ {\footnotesize $[2]=\{2,10\}$}};
                \draw(x3)[fill=black] circle(15pt) node[right]{ {\footnotesize $[3]=\{3,9\}$}};
                \draw(x4)[fill=black] circle(15pt) node[left]{ {\footnotesize $[4]=\{4,8\}$}};
                \draw(x6)[fill=black] circle(15pt) node[right]{ {\footnotesize $[6]=\{6\}$}};

                \draw (x0) -- (x2);
                \draw (x2) -- (x4);
                \draw (x4) -- (x1);
                \draw (x1) -- (x6);
                \draw (x6) -- (x3);
                \draw (x6) -- (x2);
                \draw (x3) -- (x0);
			\end{scope}
		\end{tikzpicture}
		\caption{Lattice diagram of $\mathbb{Z}_{12}$}
		\label{fig:S_3^4}
	\end{center}
\end{figure}
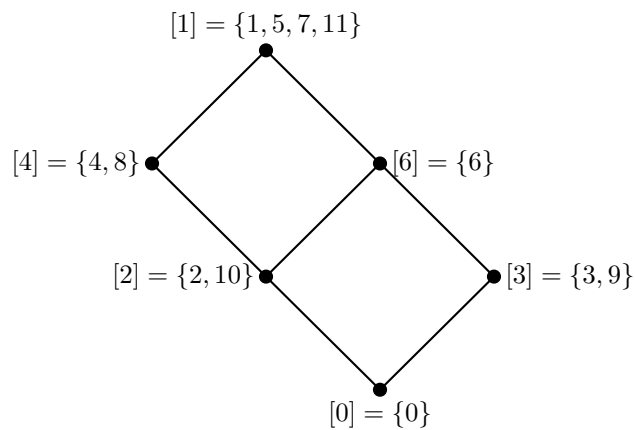

 The maximal chains $\{0, 2, 10, 4, 8, 1, 5, 7, 11\}, \{0, 2, 10, 6, 1, 5, 7, 11\}, \{0, 3, 9, 6, 1, 5, 7, 11\}$ are maximal cliques in  $\uend(\mathbb{Z}_{12})$, corresponding to the sequences $(3, 2, 2), (2, 3, 2)$ and $(2, 2,3)$ of  primes in the proof above. The largest clique has size $\phi(1)+\phi(3)+\phi(6)+\phi(12)
=1+2+2+4=9$.		
\end{eg}

\begin{rem}\rm
The clique number of the power graph of a cyclic group was determined by Alireza et al.\ \cite[Theorem 7]{Alireza} and further discussed by Kumar et al. in their survey \cite[Section 8.2]{Ajay}.    
\end{rem}

We end this section with a general result.

\begin{thm}\label{iso}
The directed endomorphism graphs (and hence the endomorphism graphs) of isomorphic groups are isomorphic. The converse is not true for endomorphism graphs.
\end{thm}
\begin{proof}
Consider two isomorphic groups, $G_1$ and $G_2$. Let $\phi:G_1 \rightarrow G_2$ be an isomorphism with $\phi(a_1)=a_2$ and $\phi(b_1)=b_2$ for some $a_1, b_1 \in G_1$ and $a_2, b_2 \in G_2$. Consider the bijection $\psi : V(\dend(G_1)) \rightarrow V(\dend(G_2))$ defined as $\psi(a)=\phi(a)$, for every $a \in V(\dend(G_1)) $. There is a directed arc from $a_1$ to $b_1$ in $\dend(G_1)$ if and only if there is an endomorphism $f$ defined on $G_1$, such that $f(a_1) = b_1$. Consider the endomorphism $\phi f \phi ^{-1}$ defined on $G_2$.
\begin{equation*}
\phi f \phi^{-1}(a_2)= \phi f(a_1) = \phi(b_1)= b_2
\end{equation*}
Hence, there is a directed arc from $a_2$ to $b_2$ in $\dend(G_2)$. The reverse implication can be proved by considering an endomorphism `$g$' on $G_2$ such that $g(a_2) = b_2$.

Now, consider the non-isomorphic groups of order 4, the cyclic group $Z_4$ under addition $+_4$ and the Klein four group. The endomorphism graph of both these groups is the complete graph $K_4$. Therefore, the converse is not true for endomorphism graphs. 
\end{proof}

\section{Abelian Groups}
\begin{thm}
If $G$ is an abelian group, then $\dend(G)$ has a single point basis.
\end{thm}
\begin{proof}
Let $G \simeq \mathbb{Z}_{n_1} \times \mathbb{Z}_{n_2} \times \ldots \times \mathbb{Z}_{n_k}, \, n_i \in \mathbb{N}$ for $1 \leq i \leq k$. Consider the element $(1, 1, \ldots ,1)$ and the endomorphism $f$ on $G$
\begin{equation*}
f(x_1, x_2, \ldots x_k) = (a_1x_1, a_2x_2, \ldots , a_kx_k)
\end{equation*}
where $a_i \in \mathbb{Z}$ for $1 \leq i \leq k$. We can see that $(1, 1, \ldots, 1)$ can be mapped to any other element using these maps.
\end{proof}

\begin{rem}
The converse of the above theorem is not true. Consider the quaternion group, $Q_8 =\langle i, j \, | \, i^8=1, j^2 = i^4, i^{-1}j i = j^{-1}\rangle$, the element $i$ can be mapped to all other elements in $Q_8$ via an endomorphism. Therefore, the converse is not true.    
\end{rem}

\begin{thm} \label{complete}
The endomorphism graph of an abelian group $G$ is complete if and only if $G = {(\mathbb{Z}_{p^a})^m} \times (\mathbb{Z}_{p^{a+1}})^n$ for some $m,n \geq 0, a \geq 1$.
\end{thm}
\begin{proof}
Suppose that $G = {(\mathbb{Z}_{p^a})^m} \times (\mathbb{Z}_{p^{a+1}})^n$ for some $m,n \geq 0, a \geq 1$.
Let $$g_1= (a_1p^{b_1}, \ldots ,a_m p^{b_m}, c_1p^{d_1}, \ldots , c_np^{d_n})$$ and	$$g_2= (a_{m+1}p^{b_{m+1}}, \ldots ,a_{2m} p^{b_{2m}}, c_{n+1}p^{d_{n+1}}, \ldots , c_{2n}p^{d_{2n}}) \in G~ ,$$\\
where $a_i, c_j \not\equiv 0\!\pmod p, 0 \leq b_i \leq a, 0 \leq d_j \leq a+1 \; \forall i, j$.\\

We will show that there exists an endomorphism $\Phi$ on $G$ such that $\Phi(g_1) = g_2$ or $\Phi(g_2)= g_1$. Choose
\begin{equation*}
\begin{split}
& b=min\{b_i\}, \; 1 \leq i \leq 2m\\
& d=min\{d_j - 1\}, \; 1 \leq j \leq 2n
\end{split}
\end{equation*}
\underline{Case $1:d < b$}\\
\ni	Let the minimum be achieved at $d_k$. Then 
\begin{equation*}
d_j  - d_k \geq 0 \; \forall j \;  \text{and} \; b_i - d_k \geq 0 \; \forall i
\end{equation*}
Define $\phi_j : \mathbb{Z}_{p^{a+1}} \rightarrow \mathbb{Z}_{p^{a+1}}$ and $\psi_i : \mathbb{Z}_{p^{a+1}} \rightarrow \mathbb{Z}_{p^{a}}$ as
\begin{equation*}
\begin{split}
& \phi_j(x) = {c_k}^{-1} c_j p ^{d_j - d_k}x \hspace{0.25cm} \forall j \; \text{and} \; \forall x \in \mathbb{Z}_{p^{a+1}}\\
& \psi_i(x) = {c_k}^{-1} a_i p ^{b_i - d_k}x \hspace{0.25cm} \forall i \; \text{and} \;\forall x \in \mathbb{Z}_{p^{a+1}}
\end{split}	
\end{equation*}
Note that $\phi_j(c_k p^{d_k})= c_j p^{d_j}$ and $\psi_i(c_k p^{d_k}) = a_i p^{b_i} \; \forall i,j$.\\
Without loss of generality assume that $k>n$. \\
Let $\pi_1: G \rightarrow \mathbb{Z}_{p^{a+1}}$ be the projection map such that $\pi_1(x_1,\ldots , x_m, y_1, \ldots, y_n) = y_{k-n}$. Then $\Phi=(\psi_1 \circ \pi _1, \ldots , \psi_m \circ \pi _1, \phi_1 \circ \pi _1, \ldots, \phi_n \circ \pi _1) : G \rightarrow G$ maps $g_2$ to $g_1$.\\
		
\ni
\underline{Case $2 :b \leq d$}\\
		
\ni
Let the minimum be achieved at $b_k$. Then 
\begin{equation*}
b_i  - b_k \geq 0 \; \forall j \;  \text{and} \; d_j - b_k \geq 1 \; \forall i
\end{equation*}
Define $\phi_i : \mathbb{Z}_{p^{a}} \rightarrow \mathbb{Z}_{p^{a}}$ and $\psi_j : \mathbb{Z}_{p^{a}} \rightarrow \mathbb{Z}_{p^{a+1}}$ as
\begin{equation*}
\begin{split}
& \phi_i(x) = {a_k}^{-1} a_i p ^{b_i - b_k}x \; \forall i \hspace{0.25cm} \text{and} \; \forall x \in \mathbb{Z}_{p^{a}} \\
& \psi_j(x) = {a_k}^{-1} c_j p ^{d_j - b_k}x \; \forall j \hspace{0.25cm} \text{and} \;  \forall x \in \mathbb{Z}_{p^{a}}
\end{split}	
\end{equation*}
Note that $\phi_i(a_k p^{b_k})= a_i p^{b_i}$ and $\psi_j(a_k p^{b_k}) = c_j ^{d_j} \; \forall i,j$. Without loss of generality assume that $k>m$. Let $\pi_2: G \rightarrow \mathbb{Z}_{p^{a}}$ be the projection map such that $\pi_2(x_1,\ldots , x_m, y_1, \ldots, y_n) = x_{k-m}$. Then $\Phi=(\phi_1 \circ \pi _2, \ldots , \phi_m \circ \pi _2, \psi_1 \circ \pi _2, \ldots, \psi_n \circ \pi _2) : G \rightarrow G$ maps $g_2$ to $g_1$.\\
		
\ni
For the converse part, if $|G|$ contain at least two prime divisors, say $p$ and $q$, then there exists elements of order $p, q$ in $G$. But $\uend(G)$  cannot be complete since no `$p$' order element can be mapped to a $`q$' order element. Also, if $G \simeq \mathbb{Z}_{p^a} \times \mathbb{Z}_{p^b} \times G_1$ where $a> b+1, b \geq 1$ and $G_1$ is any abelian $p$-group, we will prove that $\uend(G)$ is never a complete graph.\\
Take $g_1 =(p, 0, 0)$ and $g_2 = (0, 1, 0)$. Then $|g_1|= p^{a-1}$ and $|g_2| = p^b$. Since $a-1 >b, g_2 $ cannot be mapped to $g_1$.\\
Suppose there exists a map $\Phi$ that sends $g_1$ to $g_2$. Let $\Pi : G \rightarrow \mathbb{Z}_{p^b}$ be the projection map. Define $\Psi = \Pi \circ \Phi : G \rightarrow G$ as
\begin{equation*}
\Psi(x, y, z) = cx + dy + \eta(z)	
\end{equation*}
where $c$ and $d$ are some constants chosen depending on where $1$ goes and $\eta : G_1 \rightarrow \mathbb{Z}_{p^b}$.
\begin{equation*}
\Psi(p,0, 0) = c \cdot p + d \cdot 0 + \eta(0) \equiv 1\!\!\pmod{p^b}
\end{equation*}
Therefore, $p^b$ divides $cp-1$, which is a contradiction. So, there does not exist any endomorphism on $G$ mapping $g_1$ to $g_2$ or $g_2$ to $g_1$.
\end{proof}

\begin{thm}\label{order}
Let $G$ be a finite abelian group. For $a, b \in G$, if there is an endomorphism mapping $a$ to $b$, then $|b|$ divides $|a|$. The converse holds if and only if $G \simeq \prod_{i = 1}^{k} {(\mathbb{Z}_{{p_i}^{n_i}})}^{m_i}$ for distinct primes $p_i$ and $n_i, m_i \in \mathbb{N}$.
\end{thm}
\begin{proof}
The order of the image of an endomorphism always divide the order of the preimage \cite{df}.\\
Suppose that $G \simeq \prod_{i = 1}^{k} {(\mathbb{Z}_{{p_i}^{n_i}})}^{m_i}$. Let $a=(a_1, a_2, \ldots a_k)$ and $b=(b_1, b_2, \ldots b_k) \in G$, where $a_i, b_i \in {(\mathbb{Z}_{{p_i}^{n_i}})}^{m_i}$ for $ 1 \leq i \leq k$ and $|b|$ divides $|a|$. Assume $|a_i|={p_i}^{c_i}$ and $ |b_i|= {p_i}^{d_i}$ for each $i$ between $1$ and $k$. Then, we have
\begin{equation*}
    |a| = {p_1}^{c_1} {p_2}^{c_2} \ldots {p_k}^{c_k} \qquad |b| = {p_1}^{d_1} {p_2}^{d_2} \ldots {p_k}^{d_k}
\end{equation*}
where $d_i \leq c_i \leq n_i$ for $1 \leq i \leq k$. 
Consider the proof of Theorem \ref{complete}. We can see that for $a_i, b_i \in {(\mathbb{Z}_{{p_i}^{n_i}})}^{m_i}$, whenever $|b_i|$ divides $|a_i|$, there exists an endomorphism mapping $a_i$ to $b_i$.
Therefore, if $|b|$ divides $|a|$, there exists an endomorphism mapping each $a_i$ to $b_i$ for $ 1 \leq i \leq k$ ($\because \; {p_i}^{d_i}$ divides ${p_i}^{c_i}$ for each $i$) and hence there is an endomorphism mapping $a$ to $b$.\\
For the converse part assume that $G=\mathbb{Z}_{p^{u+v}} \times \mathbb{Z}_{p^{u}} \times G_1$ where $p$ is any prime $u,v \geq 0$ and $G_1$ is any abelian group. Consider the elements $a=(p^v,0,0)$ and $b=(0, 1, 0)$. Here $|a|=|b|$, but there is no endomorphism mapping $a$ to $b$(for further reading refer section `Degeneracy and ideals' in \cite{gg}).
\end{proof}

\begin{thm}\label{planar}
Let $G$ be a group and suppose that there exists an $a \in G$ such that $|G:C(a)| > 3$, where $C(a)$ denotes the centralizer of $a$, then $\uend(G)$ is non-planar.
\end{thm}
\begin{proof}
Consider the homomorphism
\begin{equation*}
f(a)=x^{-1}ax
\end{equation*}
We will show that $a - f(a) - f^2(a) - f^3(a) - e$ is a directed path in $\dend(G)$.\\
\textbf{Claim :} $f^i(a) \neq f^j(a)$ for $0 \leq i < j \leq 3$
\begin{itemize}
\item $x^m \notin C(a) \implies x^d \notin C(a)$ for any $d|m$.
\item $x^m \notin C(a) \implies f^l(a) \neq f^{l-m}(a) \; \forall \; l \geq m$
\end{itemize}
Let $x, x^2, x^3 \notin C(a)$ (such an `$x$' exists by our assumption). From the above results, we have
\begin{equation*}
\begin{split}
& a \neq f(a), a \neq f^2(a), a \neq f^3(a)\\
& f(a) \neq f^2(a), f(a) \neq f^3(a)\\
& f^2(a) \neq f^3(a)
\end{split}	
\end{equation*}
Hence we will get directed path on $5$ vertices in $\dend(G)$ and this will result in an induced $K_5$ in $\uend(G)$.
\end{proof}
	
\begin{thm}
Let $G$ be an abelian group. $\uend(G)$ is planar if and only if $|G| \leq 4$.
\end{thm}
\begin{proof}
Endomorphism graph of of $\mathbb{Z}_2, (\mathbb{Z}_2)^2, \mathbb{Z}_3$, and $\mathbb{Z}_4$ are $K_2, K_4, K_3$ and $K_4$, respectively and hence they are planar.\\	
 If there exists an $a \in G$, with $|a| >4$, then $\uend(G)$ is non-planar. Since $e - a^{-1} - a - a^2 - {(a^2)}^{-1} - e$ is a directed cycle in $\dend(G)$ and composition of a homomorphisms is again a homomorphisms, there is an induced $K_5$ in $\uend(G)$ . Therefore, $\uend(G)$ is non-planar.\\
 Suppose $\uend(G)$ is planar. Then the order of each vertex of $G$ will be less than $4$. So $G$ will be isomorphic to $\mathbb{Z}_2^a \times \mathbb{Z}_3^b \times \mathbb{Z}_4^c$ for some $ a, b, c \in \mathbb{N}$.\\	
\ni		\underline{Case $1 : b >0$}\\
If $a \neq 0$ or $c \neq 0$, then we get a $g \in G$ with $|g| \geq 4$. So $G \simeq \mathbb{Z}_3^b$. From Theorem \ref{complete} $\uend(G)$ will be $K_{3^b}$. However, since $\uend(G)$ is planar, $b=1$.\\
\ni
\underline{Case $2 : b=0$}\\
Here $G \simeq \mathbb{Z}_2^a  \times \mathbb{Z}_4^c$. Again from Theorem \ref{complete} $\uend(G)$ will be $K_{2^a 4^c}$. Therefore, $\uend(G)$ is planar in this case if and only if $a=0, \; c=1$ or $a=1 \; \text{or} \; 2, \; c=0$.
\end{proof}
	
\begin{prop}
Let $G$ be a non-trivial group. If $G \neq \mathbb{Z}_2$, then the girth, $g(\uend(G))=3$. Moreover, $\uend(G)$ is bipartite if and only if $G = \mathbb{Z}_2$.
\end{prop}
\begin{proof}
If $|G|>2$, then there exists a non-trivial automorphism, say $f$. Therefore, there are two distinct elements $a$ and $b$ different from $e$ such that $f(a) =b$, so that $e -a - b -e$ induces a cycle of length $3$ in $\uend(G)$.\\
Since there always exists the trivial homomorphism that maps each element of the group to identity, identity element alone forms a partite set. The existence of an edge between any two non-identity elements would imply that the graph is not bipartite.
\end{proof}

\begin{cor}
Endomorphism graph of a group $G$ is a tree if and only if $G = \mathbb{Z}_2$.
\end{cor}
\subsection{Abelian $\mathbf{p}$-groups}
    \label{s:abp}
    For a prime number $p$, a $p$-group is group in which order of every element is a power of $p$. Consider the abelian $p$-group $G= (\mathbb{Z}_{{p}^{m_1}})^{l_1} \times (\mathbb{Z}_{{p}^{m_2}})^{l_2} \times \ldots \times (\mathbb{Z}_{{p}^{m_k}})^{l_k}$ where $n_i, m_i \in \mathbb{N}, m_1< m_2< \ldots< m_k $. Let $|G|=n$.
   
    \subsubsection{$(\mathbb{Z}_{p^m})^l$ under addition $(+_{p^m}, +_{p^m}, \ldots +_{p^m})$}
    From Theorem \ref{complete}, we have $\uend((\mathbb{Z}_{p^m})^k)$ is complete. However, let us determine the automorphism classes of this group. \cite{gg} presents a necessary and sufficient condition for an $n \times n$ matrix over integers to be associated with an automorphism.\\
    
    \noindent
    \underline{Claim $1$: $(p^{i},0, \ldots , 0)$ can be assigned to any element of order $p^{m-i}$.}\\
    \noindent
    Let $a =(a_1, a_2, \ldots, a_l) \in (\mathbb{Z}_{p^m})^l$ with $|a|=p^{m-i}$. We can represent $a$ as $(c_1p^{j_1},c_2 p^{j_2}, \ldots c_l p^{j_l})$ where $(c_r, p) =1, i \leq j_r$ for $ 1 \leq r \leq l$ and $j_r =i$ for at least one $r$. Without loss of generality, assume that $j_1 =i$. Since $i \leq j_r \; \forall \; r$, there exist an $h_r$ such that $p^{j_r} = p^{i + h_r}$ for $2 \leq r \leq l$.  Therefore we can rewrite $a$ as 
    \begin{equation}\label{a}
    	a=(c_1 p^{i}, c_2 p^{i + h_2}, \ldots , c_l p^{i +h_l})
    \end{equation}
    Consider the matrix 
    \begin{equation*}
    		A =
    	\begin{bmatrix}
    		c_1 & 0 & 0 &\ldots & 0\\
    		c_2 p^{h_2} & 1 & 0 & \ldots & 0\\
    		c_3 p^{h_3} & 0 & 1 & \ldots & 0\\
    		\vdots\\
    		c_l  p^{h_l}& 0 & 0 & \ldots & 1\\
    	\end{bmatrix}
    \end{equation*}
    If we denote $a$ in equation \ref{a} as a column vector, then
    \begin{equation*}
    	A{[p^{i} \; 0 \; 0 \ldots 0]}^T = {[c_1 p^{i}, c_2 p^{i + h_2}, \ldots , c_l p^{i +h_l}]}^T
    \end{equation*}
    Note that even if $j_m = i$, for some $ m \neq 1$, we can transform it to $j_1 = i$ using a permutation matrix.\\
    
     \noindent
    \underline{Claim $1$: Any element of order $p^{m-{i_1}}$ can be mapped to an element of order $p^{m-{i_2}}$}\\ \underline{whenever $i_1 \leq i_2$.}\\
    Since $i_1 \leq i_2$, there exists an $i_3$ such that $p^{i_2} = p^{i_1 + i_3}$. Consider the matrix 
    \begin{equation*}
    	A_{12}=
    \begin{bmatrix}
    	p^{i_3} & 0 & 0 &\ldots & 0\\
    	0 & 1 & 0 & \ldots & 0\\
    	0 & 0 & 1 & \ldots & 0\\
    	\vdots\\
    	0 & 0 & 0 & \ldots & 1\\
    \end{bmatrix}
    \end{equation*} 
    $A_{12}$ is the matrix associated with the homomorphism that maps $(p^{i_1}, 0, \ldots , 0)$ to $(p^{i_2}, 0 , \ldots 0)$.
    Let $[p^{i}] = \{a \in \mathbf{\mathbb{Z}_{p^m}^l} \; | \; |a| = p^{m-i}\}$. Then the automorphism classes of $\mathbf{\mathbb{Z}_{p^m}^l}$ are
    \begin{equation*}
    	\Big \{[1],\; [p], \; [p^2], \ldots , [p^{m-1}] \Big\}
    \end{equation*}
    
    \subsubsection{$\mathbb{Z}_{{p}^{m_1}}^{l_1} \times \mathbb{Z}_{{p}^{m_2}}^{l_2} \times \ldots \times \mathbb{Z}_{{p}^{m_k}}^{l_k}$}
    Consider the groups  $G= \mathbb{Z}_{{p}^{m_1}}^{l_1} \times \mathbb{Z}_{{p}^{m_2}}^{l_2} \times \ldots \times \mathbb{Z}_{{p}^{m_k}}^{l_k}$ and $G'= \mathbb{Z}_{{p}^{m_1}} \times \mathbb{Z}_{{p}^{m_2}} \times \ldots \times \mathbb{Z}_{{p}^{m_k}}$  where $n_i, m_i \in \mathbb{N}, m_1< m_2< \ldots< m_k $. Let $(a_1, a_2, \ldots a_k) \in G'$. The diagonal function maps $(a_1, a_2, \ldots, a_k)$ to $(\underbrace{a_1, a_1, \ldots a_1}_{l_1}, \underbrace{a_2, a_2, \ldots a_2}_{l_2}, \ldots , \underbrace{a_k, a_k, \ldots a_k}_{l_k})$ in $G$.\\
    We have already shown that $\mathbb{Z}_{p^m}^l$ is complete. So, for examining $\dend(G)$, it is enough to study about $\dend(G')$.\\
    Let $b =(p^{b_1}, p^{b_2}, \ldots , p^{b_k})$ and $c=(p^{c_1}, p^{c_2}, \ldots , p^{c_k}) \in G'$. Assume $|c|$ divides $|b|$.\\
    \underline{Claim: There exist a homomorphism mapping $b$ to $c$ if and only if $\forall i, \exists j$ such that} \\ 
    \underline{$b_j + \max\{0, m_i - m_j\} \leq c_i$}.\\
    Let $E$ be the matrix associated with the homomorphism that maps $b$ to $c$ and $(e_1, e_2, \ldots , e_k)$ be the $i ^{th}$ row of $E$. Now from \cite{gg}
    \begin{equation*}
    	\begin{split}
    		&e_j \equiv 0(mod \,  p^{m_i - m_j}), \; \text{if} \; j \leq i\\
    		& e_j \in \mathbb{Z}_{p^{m_j}} ,\; \text{if} \; j>i
    	\end{split}
    \end{equation*}
    \begin{equation*}
    	e_1 p^{b_1}  + e_2 p^{b_2} + \ldots + e_k p^{b_k} = p^{c_i}
    \end{equation*}
    Define $f_j$ for $ 1 \leq j \leq i$ with $(f_j, p) =1$. Then 
    \begin{equation*}
    	f_1 p^{m_i - m_1}p^{b_1}  + f_2 p^{m_i - m_2}p^{b_2} +\ldots + f_i p^{b_i} + e_{i+1}p^{b_i +1} + \ldots + e_k p^{b_k} = p^{c_i}
    \end{equation*}
   or we can say
   \begin{equation*}
   	\begin{split}
   	p^{c_i} &=\sum_{j=1}^{n} d_j p^{max\{0, m_i - m_j\} + b_j} \\
   	d_j &=\begin{cases}
   		f_j & 1 \leq j \leq i\\
   		e_j & j > i
   	\end{cases} 	
   	\end{split}
   \end{equation*}	
   Hence  there exist  a homomorphism that maps $b$ to $c$ if and only if there exists a $j$ such that $\max\{0, m_i - m_j\} +b_j \leq c_i \; \forall i $.

\section{Identity element deleted subgraphs}

In endomorphism graphs the identity element $`e$' serves as a universal vertex and in the case of directed endomorphism graphs there is a directed arc to $`e$' from all other vertices. A similar nature (either serves as a universal vertex or serves as an isolated vertex) is observed for the vertex corresponding to identity element in various types of graphs defined from groups. So, at various situations, the presence of vertex corresponding to identity element is not interesting and hence it is a common practice to consider the graph induced by group elements other than the identity element. In this section, $\dende(G)$ and $\uende(G)$, denotes the induced subgraph obtained from $\dend(G)$ and $\uend(G)$, respectively, by deleting the vertex corresponding to the identity element.\\

\noindent
$\dend(G)$ is not diconnected, since there is no endomorphism mapping the identity element to any other element.
\begin{prop}
For $\dende(G)$ the following are equivalent:
\begin{enumerate}[(i)]
\item $\dende(G)$ is diconnected.
\item $\dende(G)$ is a complete digraph.
\item $\dende(G)$ is Hamiltonian.
\end{enumerate}
\end{prop}
\begin{proof}
    The equivalence follows immediately from the fact that composition of two homomorphisms is again a homomorphism.
\end{proof}
\begin{thm}
Let $G$ be an abelian group. Then $\dende(G)$ is diconnected if and only if $G \simeq (\mathbb{Z}_p)^k$ under addition $\underbrace{(+_p, +_p, \ldots , +_p)}_k$, for some prime $p$ and $k \in \mathbb{N}$.
\end{thm}
\begin{proof}
Suppose that $\dende(G)$ is diconnected, then order of each vertex divides the order of the other. So order of all the vertices will be same, say $`m$'. Let $m$ be composite and $p$ be a prime dividing $m$, then there exists an element say $`x$' with order $p$. But $x^p$ have order $\frac{n}{p}$, which is a contradiction. Therfore $m$ must be a prime number. Since $G$ is an abelian group, $G \simeq (\mathbb{Z}_p)^k$, for some prime $p$ and $k \in \mathbb{N}$.\\		
The converse holds by proof of Theorem \ref{order}.
\end{proof}

\begin{thm}
Given a group $G$, $\uende(G)$ is a tree if and only if $G = \mathbb{Z}_2$ or $\mathbb{Z}_3$ under $+_2$ and $+_3$, respectively.
\end{thm}
\begin{proof}
If $G= \mathbb{Z}_2$ or $\mathbb{Z}_3$, then $\uende(G)$ is obviously a tree. Now assume that $\uende(G)$ is a tree. Suppose that there exists a bi-directional edge between any two vertices of $\dende(G)$. Existence of a third vertex in $\uende(G)$ would imply that some vertex is adjacent to either of the previously defined vertices. Since composition of two homomorphisms is again a homomorphism, we will get a triangle in $\uende(G)$, a contradiction. With this idea we prove the theorem in three cases.

\begin{figure}[H]
\begin{minipage}{0.45\textwidth}
	\begin{center}
\begin{tikzpicture}[scale=0.7,style=thick,x=1cm,y=1cm]
			\def\vr{5pt}
			\begin{scope}[xshift=10cm, yshift=10cm] 	
                \coordinate(a) at (-1,0);
				\coordinate(b) at (0, 1.73);
				\coordinate(c) at (1,0);

                \draw(a)[fill=black] circle(4pt);
                \draw(b)[fill=black] circle(4pt);
                \draw(c)[fill=black] circle(4pt);

                \draw[->, dotted, >=stealth, line width=1pt] (c) -- (b);
                \draw[->, >=stealth, line width=1pt] (a) -- (b);
                \draw[<->, >=stealth, line width=1pt] (a) -- (c);
			\end{scope}
		\end{tikzpicture}
		
	\end{center}
  \end{minipage}
  \hfill
  \begin{minipage}{0.45\textwidth}
    \centering
   \begin{tikzpicture}[scale=0.7,style=thick,x=1cm,y=1cm]
			\def\vr{5pt}
			\begin{scope}[xshift=10cm, yshift=10cm] 				
                \coordinate(a) at (-1,0);
				\coordinate(b) at (0, 1.73);
				\coordinate(c) at (1,0);

                \draw(a)[fill=black] circle(4pt);
                \draw(b)[fill=black] circle(4pt);
                \draw(c)[fill=black] circle(4pt);

                \draw[->, dotted, >=stealth, line width=1pt] (c) -- (b);
                \draw[->, >=stealth, line width=1pt] (a) -- (b);
                \draw[<->, >=stealth, line width=1pt] (a) -- (c);
			\end{scope}
		\end{tikzpicture}
  \end{minipage}
  \caption{Formation of a cycle in $\uende$}
		\label{fig:tree}
\end{figure}
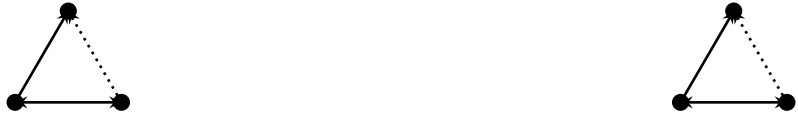

\ni\textbf{Case 1:} $G$ is abelian and $\; \text{there exists} \; \; a \in G$ with $|a| > 2$.\\
There is an arc from $a$ to $a^{-1}$ and vice versa in $\uende(G)$. The existence of a third vertex would imply $\uende(G)$ is not a tree. So $G= \mathbb{Z}_3$.\\

\ni\textbf{Case 2:} $G$ is abelian and $|a|=2 \; \text{for all} \; a \in G$.\\
Let the number of elements in $G$ be $2^n, n \in \mathbb{N}$. Then $G \cong (\mathbb{Z}_2)^n.$ From Theorem \ref{complete}, we have $\uende(G)$ is a complete graph. Therefore, $G$ must be $\mathbb{Z}_2$.\\

\ni\textbf{Case 3:} $G$ is non-abelian.\\
Since $G$ is non-abelian, $Z(G) \neq G$. Let $x \in G$ such that $x \notin Z(G)$, then there exists a $b\in G$ such that $xb \neq bx$.\\
$\phi(a)= x^{-1}a x \; \forall a \in G$ is an isomorphism on $G$. So there is an edge from $b$ to $x^{-1}bx$ and vice versa in $\dende(G)$. Therefore, we get $G$ has only 3 elements including identity. But the smallest non-abelian group has $6$ elements. Therefore, if $G$ is non-abelian, then $\uende(G)$ is not a tree.
\end{proof}

\section{Compressed endomorphism digraph of dihedral, dicyclic, symmetric and metacyclic groups}
\subsection{Dihedral Group}
    \label{s:dih}
	\begin{equation*}
		D_{2n} =\langle r, s \, | \, r^n=1, s^2 =1, srs^{-1} = r^{-1}\rangle 
	\end{equation*}
    First let us determine the automorphism classes. Any reflection can be mapped to any other reflection via an automorphism. Let $[s]$ denote the automorphism class of reflections.\\
     For any divisor $d$ of $n$, $[r^d] =\{r^e \, | \, (d, n) = (e, n)\}$, forms the remaining automorphism classes. Note that the collection of all automorphism classes of $D_{2n}$, excluding $[s]$ is similar to the automorphism classes of $\mathbb{Z}_n$. \\
     $\langle r^{\frac{n}{d}}\rangle $, where $d$ divides $n$, is a normal subgroup of $D_{2n}$. Let $\langle r^{\frac{n}{d}}\rangle $ be the kernel of some endomorphism on $D_{2n}$. $D_{2n}/\langle r^{\frac{n}{d}}\rangle $ has $\phi(m)$ cosets of order $m$, for each $m$ dividing $\frac{n}{d}$ and $\frac{n}{d}$ cosets of order $2$. Without further examinations, we get
     \begin{equation*}
     	D_{2n}/\langle r^{n/d}\rangle  \; \simeq \; D_{2{\frac{n}{d}}}
     \end{equation*}
	$f(r) = r^d$ and $f(s) =s$ is a homomorphism with kernel $\langle r^{\frac{n}{d}}\rangle $ and image set isomorphic to $D_{2\frac{n}{d}}$. Hence, there exist an edge from $[r]$ to $[r^d]$ for any $d$ dividing $n$. Now for $d$ dividing $k$, the homomorphism $f(r) =r^{\frac{k}{d}}$ and $f(s)=s$ maps $[r^d]$ to $[r^k]$. 
    The compressed directed endomorphism graph of rotations in $D_{2n}$ is same as the compressed directed endomorphism graph of $\mathbb{Z}_n$. Now let us determine the remaining edges. \\
    \underline{Case $1$} : when $n$ is odd \\
    The only normal subgroups of $D_{2n}$ are $\langle r^d\rangle $ where $d$ divides $n$ and $D_{2n}$ itself.  No rotation can be mapped to a reflection or vice versa since their orders are co-prime.\\
    \underline{Case $2$} : when $n$ is even\\
    The normal subgroups of $D_{2n}$  are $\langle r^d\rangle $ where $d$ divides $n,\; \langle r^2, \;s\rangle ,\;
     \langle r^2,\; rs\rangle $ and $D_{2n}$ itself. Let $f$ be an endomorphism on $D_{2n}$. We will find the image of the homomorphism $f$ when the following normal subgroups are kernels:
    \begin{enumerate}[(i)]
    	\item $\Ker(f) =\langle r\rangle $.\\
    	$D_{2n} / \langle r\rangle  \; \simeq \; \mathbb{Z}_2$. The possible choices for $f$ are
    	\begin{itemize}
    		\item $f(r)=e, f(s) = r^{\frac{n}{2}}$
    		\item $f(r)=e, f(s)= r^as, \; 0 \leq a \leq n-1$
    	\end{itemize}
    	\item $\Ker(f) =\langle r^2, s\rangle  =\{e, r^2, r^4, \ldots , r^{n-2}, s, r^2s, \ldots , r^{n-2}s\}$.\\
    	$D_{2n} / \langle r^2, s\rangle  \; \simeq \; \mathbb{Z}_2$. The possible choices for $f$ are
    	\begin{itemize}
    		\item $f(r)= r^as, f(s) = e,  \; 0 \leq a \leq n-1 $
    	\end{itemize}
        Note that any odd multiple of `$r$' can be mapped to a reflection via this homomorphism.
    	\item $\Ker(f) =\langle  r^2, rs\rangle  =\{e, r^2, r^4, \ldots , r^{n-2}, rs, r^3s, \ldots , r^{n-1}s\}$.\\
    	$D_{2n} / \langle   r^2, rs\rangle  \; \simeq \; \mathbb{Z}_2$. The possible choices for $f$ are
    	\begin{itemize}
    		\item $f(r)= r^as, f(s) = r^as, \; 0 \leq a \leq n-1  $
    	\end{itemize}
    \end{enumerate}

\begin{figure}[H]
\begin{minipage}{0.45\textwidth}
	\begin{center}
\begin{tikzpicture}[scale=0.7,style=thick,x=1cm,y=1cm]
			\def\vr{5pt}
			\begin{scope}[xshift=10cm, yshift=10cm] 				
                \coordinate(r) at (1,4.155);
				\coordinate(r2) at (2.61, 2.97);
				\coordinate(r3) at (3.23, 1.077);
				\coordinate(r4) at (2.61, -0.82);
                \coordinate(rd) at (1, -2);		\coordinate(r6) at (-1, -2);
                \coordinate(rk) at (-2.61, -0.82);
                \coordinate(r8) at (-3.23, 1.077);
                \coordinate(r9) at (-2.61, 2.97);
                \coordinate(s) at (-1, 4.15);

                \draw(r)[fill=black] circle(4pt) node[above]{ {\footnotesize $[r]$}};
				\draw(r2)[fill=black] circle(4pt) node[right]{ {\footnotesize $[r^2]$}};
				\draw(r3)[fill=black] circle(4pt) node[right]{ {\footnotesize $[r^3]$}};
                \draw(r4)[fill=black] circle(3pt);
                \draw(rd)[fill=black] circle(4pt) node[below]{ {\footnotesize $[r^d]$}};
                \draw(r6)[fill=black] circle(3pt);
                \draw(rk)[fill=black] circle(4pt) node[left]{\footnotesize $[r^k]$};
                \draw(r8)[fill=black] circle(3pt);
                \draw(r9)[fill=black] circle(3pt);
                \draw(s)[fill=black] circle(4pt) node[above]{ {\footnotesize $[s]$}};

                \draw[->, dotted, >=stealth, line width=1pt] (r) -- (r2);
                \draw[->, dotted, >=stealth, line width=1pt] (r) -- (r3);
                \draw[->, dotted, >=stealth, line width=1pt] (r) -- (r4);
                \draw[->, >=stealth, line width=1.5pt] (r) -- (rd);
                \draw[->, dotted, >=stealth, line width=1pt] (r) -- (r6);
                \draw[->, >=stealth, line width=1.5pt] (r) -- (rk);
                \draw[->, dotted, >=stealth, line width=1pt] (r) -- (r8);
                \draw[->, dotted, >=stealth, line width=1pt] (r) -- (r9);
                \draw[->, >=stealth, line width=1.5pt] (rd) -- (rk);
                
			\end{scope}
		\end{tikzpicture}
		
	\end{center}
  \end{minipage}
  \hfill
  \begin{minipage}{0.45\textwidth}
    \centering
    \begin{tikzpicture}[scale=0.7,style=thick,x=1cm,y=1cm]
			\def\vr{5pt}
			\begin{scope}[xshift=10cm, yshift=10cm] 				
                \coordinate(r) at (1,4.155);
				\coordinate(r2) at (2.61, 2.97);
				\coordinate(r3) at (3.23, 1.077);
				\coordinate(r4) at (2.61, -0.82);
                \coordinate(rd) at (1, -2);		\coordinate(r6) at (-1, -2);
                \coordinate(rk) at (-2.61, -0.82);
                \coordinate(r8) at (-3.23, 1.077);
                \coordinate(r9) at (-2.61, 2.97);
                \coordinate(s) at (-1, 4.15);

                \draw(r)[fill=black] circle(4pt) node[above]{ {\footnotesize $[r]$}};
				\draw(r2)[fill=black] circle(4pt) node[right]{ {\footnotesize $[r^2]$}};
				\draw(r3)[fill=black] circle(4pt) node[right]{ {\footnotesize $[r^3]$}};
                \draw(r4)[fill=black] circle(3pt);
                \draw(rd)[fill=black] circle(4pt) node[below]{ {\footnotesize $[r^d]$}};
                \draw(r6)[fill=black] circle(3pt) ;
                \draw(rk)[fill=black] circle(4pt) node[left]{\footnotesize $[r^k]$};
                \draw(r8)[fill=black] circle(3pt);
                \draw(r9)[fill=black] circle(4pt) node[left]{\footnotesize $[r^{\frac{n}{2}}]$};
                \draw(s)[fill=black] circle(4pt) node[above]{ {\footnotesize $[s]$}};

                \draw[->, dotted, >=stealth, line width=1pt] (r) -- (r2);
                \draw[->, dotted, >=stealth, line width=1pt] (r) -- (r3);
                \draw[->, dotted, >=stealth, line width=1pt] (r) -- (r4);
                \draw[->, >=stealth, line width=1.5pt] (r) -- (rd);
                \draw[->, dotted, >=stealth, line width=1pt] (r) -- (r6);
                \draw[->, >=stealth, line width=1.5pt] (r) -- (rk);
                \draw[->, dotted, >=stealth, line width=1pt] (r) -- (r8);
                \draw[->, >=stealth, line width=1pt] (r) -- (r9);
                \draw[->, >=stealth, line width=1.5pt] (rd) -- (rk);
                \draw[->, >=stealth, line width=1.5pt] (s) -- (r9);
                \draw[->, >=stealth, line width=1.5pt] (r) -- (s);
                \draw[->, >=stealth, line width=1.5pt] (r3) -- (s);
                
			\end{scope}
		\end{tikzpicture}
  \end{minipage}
  \caption{Compressed endomorphism digraph of dihedral group for odd and even cases, $d$ dividing $k$}
		\label{fig:dihedral}
\end{figure}
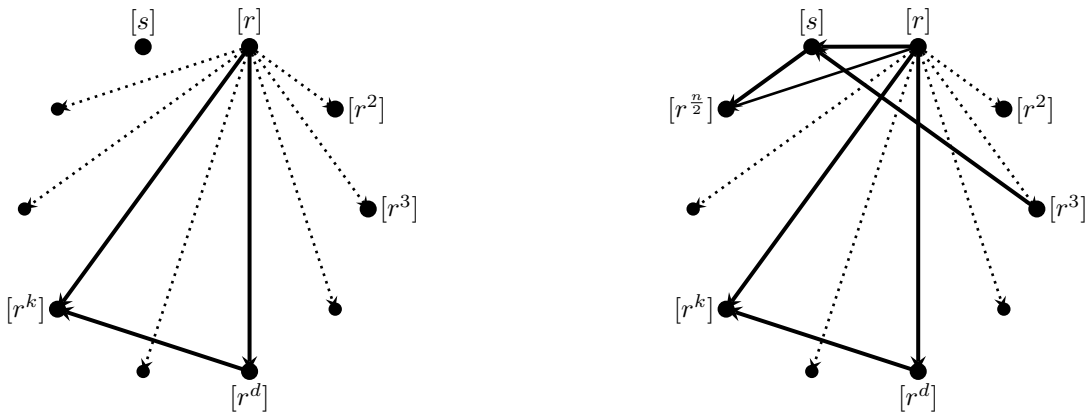

\subsection{Dicyclic group}
    \begin{equation*}
        \Dic_n = \langle a,x | a^{2n}=1, x^2 = a^n, x^{-1}ax = a^{-1} \rangle
    \end{equation*}
    For determining the automorphism classes,  consider the function $f:\Dic_n \rightarrow \Dic_n$
    \begin{equation*}
      f(a)=a^k, f(x)= a^m x\;\text{where}\;  0 \leq k,m\leq  2n-1, (k, 2n) =1 
    \end{equation*}Clearly $f$ is a homomorphism. Let $\Im(f)$ denote the image set of $f$. Note that $a \in\Im(f)$, since $a^k \in\Im(f)$. Consequently $a^m$ and $a^{-m} \in \Im(f)$. Hence $a^{-m} a^m x  = x \in\Im(f)$. Therefore, $\Im(f) =\Dic_n$. An onto homomorphism on a finite set is an isomorphism. Hence $f$ is an automorphism. The automorphism classes are :
    \begin{itemize}
        \item $[a^d] = \{a^e \, | \, (e, 2n) = (d, 2n)\}$ for $d / 2n$
        \item $[x] = \{a^m x \, | \,  0 \leq m \leq 2n-1\}$
    \end{itemize}
Having established the automorphism classes, let us find the arcs between them.
\begin{enumerate}[(i)]
\item There exist an arc from $[a]$ to $[a^d]$ for any $d /2n$ and from $[a^d]$ to $[a^k]$  for any $d/k$.  Consider the homomorphism $f$ on $\Dic_n$ given by  $f(a)= a^k$ and $f(x) =x$. $\Ker(f) =\langle a^{\frac{2n}{k}}\rangle$, which is a normal subgroup of  $\Dic_n$. $f$ maps $[a]$ to $[a^d]$. Now $f(a)= a^{\frac{k}{d}}$ and $f(x) =x$ maps $[a^d]$ to $[a^k]$  for any $d/k$.
\item Edges incident with $x$. 
\begin{itemize}
    \item $f(a)=1$ and $f(x) = a^n$ is  a homomorphism with kernel $\langle a \rangle$ and image isomorphic to $\mathbb{Z}_2$.\\
    \item Suppose $f(a) =x$ and $f(x) = a^k$. Since $|x|=4$ divides $|a|$, $n$ must be even for this case. 
    \begin{equation*}
        f(x^{-1}ax) = f(x^{-1})f(a)f(x) = a^{-k}xa^k = a^{-2k}x        
    \end{equation*}
    However
    \begin{equation*}
        f(a^{-1})=x^{-1} = a^n x
    \end{equation*}
    \begin{equation*}
        a^n x= a^{-2k}x  \iff k=\frac{n}{2}
    \end{equation*}
    Also
    \begin{equation*}
        f(x^2) = f(a^n) \;\text{implies} \; a^n =x^n
    \end{equation*}
    Since $x^2=a^n$ and $x^4=1, n$ should be even and not divisible by $4$. Therefore, there exists an arc from $[x]$ to $[a^k]$ if and only if $n \equiv 2 (mod \, 4)$. Note that the homomorphism $f(a) =x, f(x) = a^{\frac{n}{2}}$ maps $[a^d]$ to $[x]$ for $d \equiv 1(mod \, 4)$ and $f(a) =a^n x, f(x) = a^{\frac{n}{2}}$ maps $[a^d]$ to $[x]$ for $d \equiv 3(mod \, 4)$.
\end{itemize}
\end{enumerate}

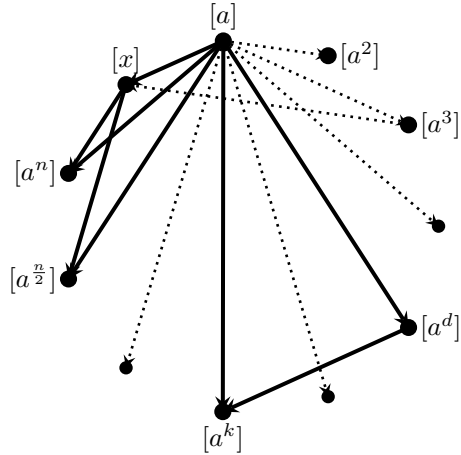
\begin{figure}[H]
\begin{center}
\begin{tikzpicture}[scale=0.7,style=thick,x=1cm,y=1cm]
			\def\vr{5pt}
			\begin{scope}[xshift=10cm, yshift=10cm] 				
                \coordinate(a) at (-1.09, 3.51);
				\coordinate(a2) at (0.880189, 3.228707);
				\coordinate(a3) at (2.391688, 1.918986);
				\coordinate(a4) at (2.955153,0);
                \coordinate(ad) at (2.3916, -1.9189);
                \coordinate(a6) at (0.880189, -3.228707);
                \coordinate(ak) at (-1.099454, -3.513337);
                \coordinate(a8) at (-2.918718, -2.682507);
                \coordinate(an2) at (-4, -1);
                \coordinate(an) at (-4, 1);
                \coordinate(x) at (-2.918718, 2.682507);

                \draw(a)[fill=black] circle(4pt) node[above]{ {\footnotesize $[a]$}};
				\draw(a2)[fill=black] circle(4pt) node[right]{ {\footnotesize $[a^2]$}};
				\draw(a3)[fill=black] circle(4pt) node[right]{ {\footnotesize $[a^3]$}};
                \draw(a4)[fill=black] circle(3pt);
                \draw(ad)[fill=black] circle(4pt) node[right]{ {\footnotesize $[a^d]$}};
                \draw(a6)[fill=black] circle(3pt);
                \draw(ak)[fill=black] circle(4pt) node[below]{\footnotesize $[a^k]$};
                \draw(a8)[fill=black] circle(3pt);
                \draw(an2)[fill=black] circle(4pt) node[left]{\footnotesize $[a^{\frac{n}{2}}]$};
                \draw(an)[fill=black] circle(4pt) node[left]{\footnotesize $[a^n]$};
                \draw(x)[fill=black] circle(4pt) node[above]{ {\footnotesize $[x]$}};

                \draw[->, dotted, >=stealth, line width=1pt] (a) -- (a2);
                \draw[->, dotted, >=stealth, line width=1pt] (a) -- (a3);
                \draw[->, dotted, >=stealth, line width=1pt] (a) -- (a4);
                \draw[->, >=stealth, line width=1.5pt] (a) -- (ad);
                \draw[->, dotted, >=stealth, line width=1pt] (a) -- (a6);
                \draw[->, >=stealth, line width=1.5pt] (a) -- (ak);
                \draw[->, dotted, >=stealth, line width=1pt] (a) -- (a8);
                \draw[->, >=stealth, line width=1.5pt] (a) -- (an2);
                \draw[->, >=stealth, line width=1.5pt] (a) -- (an);
                \draw[->, >=stealth, line width=1.5pt] (a) -- (x);\draw[->, >=stealth, line width=1.5pt] (x) -- (an);
                \draw[->, >=stealth, line width=1.5pt] (x) -- (an2);
                \draw[->, >=stealth, line width=1.5pt] (ad) -- (ak);
                \draw[->, dotted, >=stealth, line width=1pt] (a3) -- (x);

			\end{scope}
		\end{tikzpicture}		
	\end{center}

  \caption{$\dend
  _-(Dic_n)$, when $n \equiv 2 (mod \, 4)$}
		\label{fig:dicyclic}
\end{figure}

\subsection{Symmetric group}
    The symmetric group, $S_n$ is the group of all permutations on $n$ symbols. The automorphism group of $S_n$ for different cases of $n$ is listed below
    \begin{enumerate}[(i)]
        \item $n \neq 6$\\
        When $n \neq 6$, the automorphism group of $S_n$ is simply the collection of inner automorphisms, $Inn(S_n)$. For $g \in S_n,$ an inner automorphism $f_g : S_n \rightarrow  S_n$ is defined as
        \begin{equation*}
            f_g(x) = g^{-1}xg, \; \forall x \in S_n.
        \end{equation*}
        
        Therefore the automorphism classes are precisely the conjugacy classes of $S_n$. Note that elements with the same cyclic structure belong to the same conjugacy classes. We will represent the automorphism classes using the cyclic structure possessed by the permutation.
        \item $n = 6$ \\
        For $n=6$, apart from the inner automorphism, there are also outer automorphisms, $Out(S_n)$. With this outer automorphisms, elements with the following cyclic structures can be mapped together
        \begin{itemize}
            \item $(2) \longleftrightarrow {(2)}^3$
            \item $(3) \longleftrightarrow {(3)}^2$
            \item $(2)(3) \longleftrightarrow (6)$
        \end{itemize}
    \end{enumerate}
    Thus the automorphism classes of the group $S_6$ are :
    \begin{equation*}
        [e], \;[(2)] =[{(2)}^3], \; [{(2)}^2], \; [(3)] = [{(3)}^2], \; [(2)(3)] = [(6)], \; [(4)], \; [(4)(2)], \; [(5)] 
    \end{equation*}

    For identifying the endomorphisms, let us consider two cases.
    \begin{itemize}[*]
        \item Case $1 : n \neq 4$\\
        The only trivial normal subgroup of $S_n$ when $n \neq 4$, is the alternating group $A_n$. $S_n / A_n \simeq \mathbb{Z}_2$. So the image set of the endomorphism with kernel $A_n$ will be isomorphic to $\mathbb{Z}_2$. With this endomorphism, any odd permutation can be mapped to an element of order $2$ in $S_n$. Note that an odd permutation is a permutation that can be expressed as the product of an odd number of transpositions. 
        \item Case $2 : n = 4$\\
        If we consider the normal subgroup $A_4$ as the kernel of an endomorphism, automorphism class $[(4)]$ can mapped to $[{(2)}^2]$ and $[(2)]$.
        In addition to $A_4$, $V_4 = \{e, (12)(34), (13)(24), (14)(23) \}$ is a normal subgroup of $S_4$. Even if we take $V_4$ as the kernel of an endomorphism, there are still no mappings among the elements since the order of the image should divide the order of the preimage of an endomorphism.
    \end{itemize}

    \begin{figure}[ht]
\begin{minipage}{0.45\textwidth}
	\begin{center}
\begin{tikzpicture}[scale=0.9,style=thick,x=1cm,y=1cm]
			\def\vr{5pt}
			\begin{scope}[xshift=10cm, yshift=10cm] 				
                \coordinate(2) at (-0.43, 2.24);
				\coordinate(22) at (1.51, 1.80);		\coordinate(3) at (2.38, 0);
				\coordinate(42) at (1.51, -1.80);
                \coordinate(4) at (0.43, -2.24);		\coordinate(5) at (-2, -1);
                \coordinate(6) at (-2, 1);

                \draw(2)[fill=black] circle(3pt) node[above]{ {\footnotesize $[(2)]$}};
                \draw(22)[fill=black] circle(3pt) node[above]{ {\footnotesize $[(2^2)]$}};
                \draw(3)[fill=black] circle(3pt) node[right]{ {\footnotesize $[(3)]$}};
                \draw(42)[fill=black] circle(3pt) node[right]{ {\footnotesize $[(4)(2)]$}};
                \draw(4)[fill=black] circle(3pt) node[below]{ {\footnotesize $[(4)]$}};
                \draw(5)[fill=black] circle(3pt) node[left]{ {\footnotesize $[(5)]$}};
                \draw(6)[fill=black] circle(3pt) node[left]{ {\footnotesize $[(6)]$}};

                \draw[->, >=stealth, line width=1pt] (6) -- (22);
                \draw[->, >=stealth, line width=1pt] (4) -- (2);
                \draw[->, >=stealth, line width=1pt] (4) -- (22);
                \draw[->, >=stealth, line width=1pt] (6) -- (2);
                 
			\end{scope}
		\end{tikzpicture}
		
	\end{center}
  \end{minipage}
  \hfill
  \begin{minipage}{0.45\textwidth}
    \centering
    \begin{tikzpicture}[scale=0.9,style=thick,x=1cm,y=1cm]
			\def\vr{5pt}
			\begin{scope}[xshift=10cm, yshift=10cm] 				
                \coordinate(2s) at (0, 2);
				\coordinate(22s) at (2, 0);
				\coordinate(3s) at (0, -2);
				\coordinate(4s) at (-2, 0);

                \draw(2s)[fill=black] circle(3pt) node[above]{ {\footnotesize $[(2)]$}};
                \draw(22s)[fill=black] circle(3pt) node[right]{ {\footnotesize $[(2^2)]$}};
                \draw(3s)[fill=black] circle(3pt) node[below]{ {\footnotesize $[(3)]$}};
                \draw(4s)[fill=black] circle(3pt) node[left]{ {\footnotesize $[(4)]$}};

                \draw[->, >=stealth, line width=1pt] (4s) -- (2s);
                \draw[->, >=stealth, line width=1pt] (4s) -- (22s);
                
			\end{scope}
		\end{tikzpicture}
  \end{minipage}
  \caption{Compressed endomorphism digraph of $S_6$ and $S_4$}
		\label{fig:symmetric}
\end{figure}
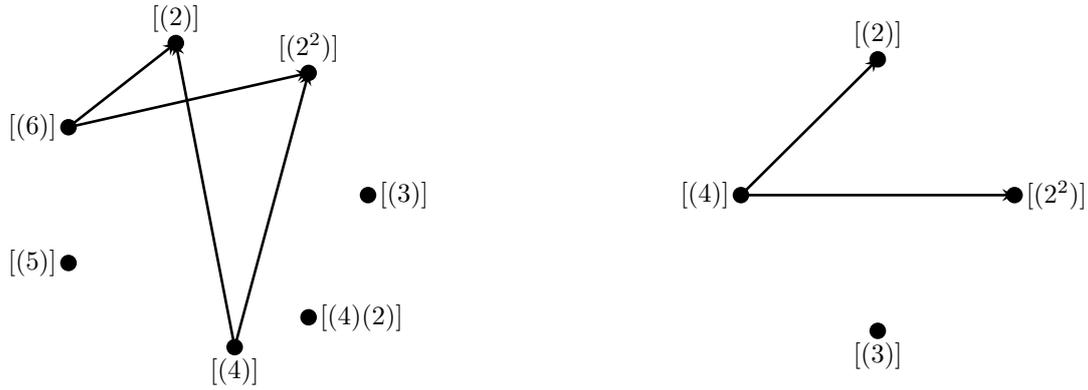
\begin{rem}\rm
        For the alternating group, $A_n$ the automorphism classes are the conjugacy classes of $S_n$ that are present in $A_n$ except for the case $n =6$. This is because
    \begin{equation*}
Aut(A_n)=\begin{cases}
    Z_1, & \text{if $n = 1,2$}\\
    Z_2, & \text{if $n = 3$}\\
    S_n, & \text{if $n \geq 4, n \neq 6$}\\
    S_6 \rtimes Z_2, & \text{if $n = 6$}
  \end{cases}
\end{equation*}
For $n=6$, apart from the conjugacy classes of $S_6$, that belong to $A_6$, a $(3)$ cycle and $(3)(3)$ cycle belong to the same automorphism class.\\
As for the endomorphisms, when $n \neq 4$, there are no non-trivial normal subgroups for $A_n$ and  hence the $\dend_-(A_n)$ is a disconnected graph. For $n=4,$ the vertex set of $\dend_-(A_4), V(A_4) =\{[(3)], [(2)(2)], e\}$. Even though, $V_4 =\{(12)(34), (13)(24), (14)(23)\}$, is a non-trivial normal subgroup of $A_4$, there do not exist any edges  between the vertices, since the order of the image should divide the order of the pre-image.
\end{rem}

\subsection{Some metacyclic groups}

Let $q$ be a fixed prime. For each divisor $m$ of $q-1$ with $m>1$, let
$G_m$ be the semidirect product of $\mathbb{Z}_q$ by $\mathbb{Z}_m$ where
the latter acts faithfully on the former; that is,
\[G_m=\langle x,y\mid x^q=y^m=1,y^{-1}xy=x^s\rangle,\]
where $s$ is a primitive $m$-th root of 1 (mod~$q$).

\begin{thm}
The compressed endomorphism graph of $G_m$ has two connected components, an
isolated vertex and the uncompressed endomorphism graph of $\mathbb{Z}_m$ with
the identity removed.
\end{thm}

The endomorphism graph of a cyclic group is described in Section~\ref{s:cyc}.

\begin{proof}
First, we establish the following claim:
\[\Aut(G_m)\cong G_{q-1}.\]
For this, we note that $G_m$ is a subgroup of $G_{q-1}$. It is a normal
subgroup by the Correspondence Theorem, since $\mathbb{Z}_m$ is a normal
subgroup of $\mathbb{Z}_{q-1}$. So the inner automorphisms of $G_{q-1}$
(which form a group isomorphic to $G_{q-1}$) act as automorphisms
of $G_m$. 

This action is transitive on pairs consisting of a non-identity power of $x$
and a subgroup of order $m$; so, if there are any further automorphisms,
we may assume that one of them, say $\sigma$, fixes $x$ and maps $y$ to
$y^j$. Then we have
\[x^s=(x^s)^\sigma=y^-\sigma x^\sigma y\sigma=y^{-j}xy^j=x^{s^j};\]
so $s^j$ is congruent to $s$ (mod~$q$). By our assumption on $m$, it follows
that $j\equiv1$ (mod~$m$), so $y^j=y$ and $\sigma$ is the identity.

So the claim is proved.

Now we can read off the isomorphism classes of $G_m$; they are
\begin{itemize}
\item[] $[x]=\{x^i:1\le i\le q-1\}$, and
\item[] $[y^j]=\{x^iy^j:0\le i\le q-1\}$ for $1\le j\le m$.
\end{itemize}

No endomorphism can map an element of order $q$ to a non-identity element
of order dividing $m$, or \emph{vice versa}; so $[x]$ is an isolated vertex
in the endomorphism graph. 

For the other isomorphism classes, there is a natural bijection to the 
non-identity elements of $\langle y\rangle\cong\mathbb{Z}_m$; and any
endomorphism of $G_m$ induces an endomorphism of $\mathbb{Z}_m$. Conversely,
if $\tau$ is an endomorphism of $\mathbb{Z}_m$, then we may assume that the
domain of $\tau$ is $G_m$ (with $\langle x\rangle$ in the kernel) and the
range is $\langle y\rangle$; then $\tau$ is an endomorphism of $G_m$. Thus
the induced subgraph of the endomorphism graph on the set of automorphism
classes of the second type is isomorphic to the uncompressed endomorphism
graph of $\mathbb{Z}_m$. This concludes the proof.
\end{proof}

\begin{cor}
If $q$ is prime and $m$ is a prime power with $m\mid q-1$, then the endomorphism graph of
$G_m$ is the disjoint union of $K_1$ and $K_{m-1}$.
\end{cor}

\section{Negative answers to Question $\mathbf{3.1}$}\label{s:examps}

In this section we present examples giving negative answers to the remaining
parts of Question~\ref{q:first}: non-isomorphic groups with isomorphic endomorphism
digraphs; groups with isomorphic endomorphism graphs but non-isomorphic
endomorphism digraphs; and groups in which the endomorphism and automorphism
classes are different.

\begin{prop}
Let $p$ be an odd prime. Then the following three groups of order $p^3$ all
have isomorphic endomorphism digraphs: the cyclic group $Z_{p^3}$; the group
$Z_{p^2}\times Z_p$; and the non-abelian group of order $p^3$ and exponent
$p^2$.
\end{prop}

\begin{proof} The endomorphism digraph of a group is a preorder; we
will show that the preorders are the same in all three cases. (The convention
is that $a>b$ if there is an endomorphism mapping $a$ to $b$.)

We let $O_n(G)$ denote the set of elements of order $n$ in $G$. 

In the cyclic group of order $p^3$, the preorder is
\[O_{p^3}(G)>O_{p^2}(G)>O_p(G)>O_1(G)=\{e\},\]
where the sets have cardinalities $p^2(p-1)$, $p(p-1)$, $p-1$ and $1$
respectively. For an element of order $p^3$ is a generator and can be mapped
to any element; thus its $p$th power can be mapped to any element of order
$p^2$ (since any such element has a $p$th root); and so on. There are no
further arcs since endomorphisms cannot increase orders of elements.

Now consider $G=Z_{p^2}\times Z_p=\langle a,b\rangle$, where $a$ and $b$
have orders $p^2$ and $p$. Now $a$ can be chosen to be any element of order
$p^2$, and $b$ any element not in $\langle a \rangle$. From this it is easy to see that there
are endomorphisms mapping elements of order $p^2$ to all elements of the
group. (Map $b$ to $e$ and $a$ to the required element.) Now elements in
$\langle a^p\rangle$ can only be mapped into this set since they are the
only elements of the group which are $p$th powers; but other elements of
order $p$ can be mapped to arbitrary elements of order $p$ or $1$. So the
preorder is
\[O(p^2)>O(p)\setminus\langle a^p\rangle>\langle a^p\rangle\setminus\{e\}
>\{e\},\]
and the sets have cardinalities $p^3-p^2$, $p^2-p$, $p-1$ and $1$. So the
graphs for these two groups are isomorphic.

Note  that this argument also holds for $p=2$, so $Z_8$ and $Z_4\times Z_2$
have isomorphic endomorphism digraphs.

Finally, consider $G=\langle a,b:a^{p^2}=b^p=e, b^{-1}ab=a^{1+p}\rangle$.
The elements of order $p$ in this group form a subgroup $H=\langle a^p,b\rangle$
of order $p^2$. Any element outside this subgroup can play the role of $a$.
So as in the previous case there are arcs from elements of order $p^2$ to
all elements of the group. Also, 
\begin{eqnarray*}
(a^ib^j)^p&=&a^i\cdot b^ja^ib^{-j}i\cdot b^{2j}a_ib^{-2j}\cdots
b^{(p-1)j}a^ib^{-(p-1)j} \\
&=& a^{pi}\cdot a^{-i(p+\cdot+2p+\cdots+(p-1)p)} \\
&=& a^{pi}\cdot a^{-p(p-1)i/2} 
\end{eqnarray*}
which is a $p$th power (for $p$ is odd, so $p$ divides $p(p-1)/2$). So again,
only the powers of $a^p$ are $p$th powers, and the argument proceeds as before.
\end{proof}

\begin{prop}
Suppose that $G$ and $H$ are groups of coprime order. Then the endomorphism
digraph of $G\times H$ is the strong product of the endomorphism digraphs
of $G$ and $H$. 
\end{prop}

For any endomorphism of $G_1\times G_2$ must be a product of endomorphisms
of $G$ and $H$; so, an endomorphism of the direct product maps
$(g_1,h_1)$ to $(g_2,h_2)$ if and only if there are endomorphisms of
$G$ mapping $g_1$ to $g_2$ and of $H$ mapping $h_1$ to $h_2$.

This gives many more examples of pairs of groups with isomorphic endomorphism
graphs.

For the next question, we use the remaining two groups of order $p^3$.

\begin{prop}
Let $p$ be  an odd prime. Let $H_p$ be the nonabelian group of order $p^3$
and exponent $p$. Then $\uend(H_p)$ is isomorphic to $\uend(\Z_p^3)$ but 
$\dend(H_p)$ is not isomorphic to $\dend(\Z_p^3)$.
\end{prop}

\begin{proof}
The automorphism group of $\Z_p^3$ is the general linear group 
$\mathrm{GL}(3,p)$, and acts transitively on the non-identity elements. So
$\dend(\Z_p^3)$ with the identity deleted is the complete directed graph,
and $\uend(\Z_p^3)$ is the complete graph.

Now let $Z$ be the center of $H_p$. Then $|Z|=p$. No automorphism can map an
element of $Z$ to an element outside $Z$. Also, $Z$ is the Frattini subgroup
of $G$, and so is contained in any nontrivial maximal subgroup; there are
$p+1$ such subgroups, corresponding to the subgroups of order $p$ in $H_p/Z$.
So all proper endomorphisms map $Z$ to the identity. Thus $\dend(H_p)$ is not
isomorphic to $\dend(\Z_p^3)$. However, there is an endomorphism of $H_p$ whose
image is $H_p/Z\cong\Z_p^2$, so any element $a$ outside $Z$ can be mapped to a
non-identity element of $\Z_p^2$. Since $H_p$ is the union of subgroups
isomorphic to $\Z_p^2$, we see that $a$ is joined to every element of $H_p$ in
$\uend(H_p)$. Thus $\uend(H_p)$ is complete, and is isomorphic to
$\uend(\Z_p^3)$.
\end{proof}

Finally, here is a family of examples of groups where automorphism and
endomorphism equivalence are not the same.

\begin{prop}
Let $q$ and $r$ be odd primes such that $r$ divides $q-1$, and let $G$ denote
the nonabelian group of order $qr$. Then all elements of order $r$ in $G$
lie in the same endomorphism class, but these elements fall into $r-1$
different automorphism classes.
\end{prop}

\begin{proof}
This group contains a normal subgroup $N$ of order $q$; so $G/N$ is cyclic of
order $r$. Every element outside $N$ has order $r$; so every element outside
$N$ can be mapped to a generator of a cyclic group of order $r$, and so to any
other element of $G$ of order $r$, by an endomorphism. So $G\setminus N$
is a single endomorphism class. But, if $h$ has order $R$, then there is a
number $m=m(h)$, an $r$th root of unity mod $q$, such that $h^{-1}xh=x^m$ for
all elements $x$ of order $p$. If an automorphism maps $h$ to $h'$, then
$m(h)=m(h')$. So there are $r-1$ automorphism classes of elements of order~$r$,
each of size $q$.
\end{proof}

\section{Miscellaneous topics and questions}

In the final section, we discuss a few additional topics and pose some
questions for further research.

\subsection{Groups with $\daut(G)=\dend(G)$}

Recall that $\daut(G)$ is the directed automorphism graph of $G$. Which
groups $G$ satisfy $\daut(G)=\dend(G)$? These groups can be divided into two
types:
\begin{enumerate}
\item groups in which every non-trivial endomorphism is an automorphism;
\item groups with non-trivial proper endomorphisms, but satisfying the
condition that, if there is an endomorphism mapping $x$ to $y$, then there is
an automorphism doing the same.
\end{enumerate}

For the first type, we have the following. Recall that $G$ is simple if its
only normal subgroups are $G$ and $\{e\}$, and $G$ is perfect if its derived
subgroup $G'$ is equal to $G$.

\begin{prop}
Let $z_G$ denote the trivial endomorphism of $G$ mapping every element to $e$.
\begin{enumerate}
\item If $G$ is simple, then $\End{G}=\Aut(G)\cup\{z_G\}$.
\item If $\End{G}=\Aut(G)\cup\{z_G\}$, then $G$ is perfect.
\item Neither of these implications reverses.
\end{enumerate}
\end{prop}

\begin{proof}
(a) The kernel of an endomorphism $f$ is a normal subgroup; so, if $G$ is
simple, then $\ker(f)$ is either $\{e\}$ (in which case $G$ is an automorphism)
or $G$ (in which case $f=z_G$).

\smallskip

(b) Suppose that $G'\ne G$. Then $G/G'$ is abelian, so we can choose a 
normal subgroup $N$ containing $G'$ such that $G/N$ has prime order $p$.
Then $p$ divides $|G|$, so by Cauchy's theorem, $G$ has a subgroup $P$ of
order $p$. Now an isomorphism from $G/N$ to $P$ lifts to an endomorphism $f$
of $G$ with $\ker(f)=N$.

\medskip

(c) The group $G_1=\mathrm{SL}(2,5)$ of $2\times2$ matrices over the field of
five elements has the property that its only proper non-trivial normal
subgroup is $Z=\{\pm I\}$; and $G_1/Z$ is isomorphic to $\mathrm{PSL}(2,5)$,
which is the simple group $A_5$. Also, if $A_5$ were a subgroup of $G_1$, then
it would have index $2$ and be normal, which is not so. So there is no 
endomorphism of $G_1$ with kernel $Z$, but $G$ is not simple.

The group $G_2=\mathrm{AGL}(3,2)$, the three-dimensional affine group over
the field of $2$ elements (generated by translations and invertible linear
maps of the $3$-dimensional vector space) has a unique proper non-trivial
subgroup $T$ consisting of the translations; the quotient $G_1/N$ is the
simple group $\mathrm{GL}(3,2)$, so $G_2$ is perfect; but $\mathrm{GL}(3,2)$
is a subgroup of $G_2$ consisting of the linear maps, so there is an
endomorphism with kernel $T$.
\end{proof}

For the second type, we have a classification of the abelian groups with this
property.

\begin{prop}
Let $G$ be a finite abelian group. Then $G$ has the property that
$\dend(G)=\daut(G)$ if and only if $G$ is an elementary abelian $p$-group
for some prime $p$.
\end{prop}

\begin{proof} If $G$ contains an element $g$ of composite order $rs$ where
$r,s>1$, then $G$ has an endomorphism (the power map $x\mapsto x^s$) which
maps $g$ to $g^s$; but no automorphism can do this, since $g$ and $g^s$
have different orders.
\end{proof}

The classification for nonabelian groups is still open.

\subsection{Endomorphism digraph and power digraph}

We remarked at the start of Section~\ref{s:cyc} that, if $G$ is a cyclic group,
then the directed endomorphism graph $\dend(G)$ and directed power graph
$\dpow(G)$ coincide. What happens in general?

\begin{prop}
If $G$ is a finite abelian group, then $\dpow(G)$ is a spanning subgraph of
$\dend(G)$. Equality holds if and only if $G$ is cyclic.
\end{prop}

\begin{proof}
The power maps $x\mapsto x^r$ for $r\in\mathbb{N}$ form a submonoid of the
endomorphism monoid of $G$; the directed edges of the power graph are 
exactly those in the digraph associated with this submonoid.

We observed that the graphs coincide for cyclic groups. Conversely, suppose
that $G$ is not cyclic. Then
\[G\cong\Z_{m_1}\times\Z_{m_2}\times\cdots\times\Z_{m_t}\]
where $m_1\mid m_2\mid \cdots \mid m_t$. The map 
\[(x_1,x_2,\ldots,x_t)\mapsto(0,x_2,\ldots,x_t)\]
is an endomorphism; but no power map can have this effect.
\end{proof}

The relationship is not so simple for nonabelian groups since not all power
maps are endomorphisms. (Indeed, it is known that if the maps $x\mapsto x^r$
are endomorphisms for any three consecutive values of $r$, then $G$ must be
abelian.)

If $G$ is abelian, we can define a ``difference digraph'' whose arcs are all
the arcs of $\dend(G)$ which are not arcs of $\dpow(G)$, and a corresponding
undirected difference graph. The properties of these graphs await investigation,

\end{document}